\documentclass[a4paper,11pt]{article}

\usepackage{amsmath,amsthm,amssymb,mathrsfs,latexsym,amsfonts}
\usepackage{graphicx,psfrag,epsfig}
\usepackage{bbm}
\usepackage[english]{babel}
\usepackage[latin1]{inputenc}
\usepackage{a4wide}

\newtheorem{theorem}{Theorem}[section]
\newtheorem{lemma}[theorem]{Lemma}
\newtheorem{proposition}[theorem]{Proposition}

\newtheorem{corollary}[theorem]{Corollary}

\newcounter{paraga}[section]
\renewcommand{\theparaga}{{\bf\arabic{paraga}.}}
\newcommand{\paraga}{\medskip \addtocounter{paraga}{1} 
\noindent{\theparaga\ } }

\begin{document}

\def\MP{\,{<\hspace{-.5em}\cdot}\,}
\def\SP{\,{>\hspace{-.3em}\cdot}\,}
\def\PM{\,{\cdot\hspace{-.3em}<}\,}
\def\PS{\,{\cdot\hspace{-.3em}>}\,}
\def\EP{\,{=\hspace{-.2em}\cdot}\,}
\def\PP{\,{+\hspace{-.1em}\cdot}\,}
\def\PE{\,{\cdot\hspace{-.2em}=}\,}
\def\N{\mathbb N}
\def\C{\mathbb C}
\def\Q{\mathbb Q}
\def\R{\mathbb R}
\def\T{\mathbb T}
\def\A{\mathbb A}
\def\Z{\mathbb Z}
\def\demi{\frac{1}{2}}

\begin{titlepage}
\author{Abed Bounemoura~\footnote{abedbou@gmail.com, Centre de Recerca Matemàtica, Campus de Bellaterra, Edifici C, 08193, Bellaterra}}
\title{\LARGE{\textbf{Normal forms, stability and splitting of invariant manifolds II. Finitely differentiable Hamiltonians}}}
\end{titlepage}

\maketitle

\begin{abstract}
This paper is a sequel to ``Normal forms, stability and splitting of invariant manifolds I. Gevrey Hamiltonians", in which we gave a new construction of resonant normal forms with an exponentially small remainder for near-integrable Gevrey Hamiltonians at a quasi-periodic frequency, using a method of periodic approximations. In this second part we focus on finitely differentiable Hamiltonians, and we derive normal forms with a polynomially small remainder. As applications, we obtain a polynomially large upper bound on the stability time for the evolution of the action variables and a polynomially small upper bound on the splitting of invariant manifolds for hyperbolic tori.
\end{abstract}
 
\section{Introduction and main results}\label{s1}

\paraga Let us briefly recall the setting considered in the first part of this work \cite{BouI12}. Let $n\geq 2$ be an integer, $\T^n=\R^n/\Z^n$ and $B_R$ be the closed ball in $\R^n$, centered at the origin, of radius $R>0$ with respect to the supremum norm. For $\varepsilon\geq 0$, we consider an $\varepsilon$-perturbation of an integrable Hamiltonian $h$ in angle-action coordinates, that is a Hamiltonian of the form
\begin{equation*}
\begin{cases} 
H(\theta,I)=h(I)+f(\theta,I) \\
|f| \leq \varepsilon <\!\!<1
\end{cases}
\end{equation*}
where $(\theta,I) \in \mathcal{D}_R=\T^n \times B_R$ and $f$ is a small perturbation in some suitable topology defined by a norm $|\,.\,|$. The phase space $\mathcal{D}_R$ is equipped with the symplectic structure induced by the canonical symplectic structure on $\T^n \times \R^n=T^*\T^n$. 

For $\varepsilon=0$, the action variables are integrals of motion and the phase space is then trivially foliated into invariant tori $\{I=I_0\}$, $I_0 \in B_R$, on which the flow is linear with frequency $\nabla h(I_0)$. Let us focus on the invariant torus $\mathcal{T}_0=\{I=0\}$. The qualitative and quantitative properties of this invariant torus are then determined by the Diophantine properties of its frequency vector $\omega=\nabla h(0) \in \R^n$. Let us say that a vector subspace of $\R^n$ is rational if it has a basis of vectors with rational (or equivalently, integer) components, and we let $F=F_\omega$ be the smallest rational subspace of $\R^n$ containing $\omega$. If $F=\R^n$, the vector $\omega$ is said to be non-resonant and the dynamics on the invariant torus $\mathcal{T}_0$ is then minimal and uniquely ergodic. If $F$ is a proper subspace of $\R^n$ of dimension $d$, the vector $\omega$ is said to be resonant and $d$ (respectively $m=n-d$) is the number of effective frequencies (respectively the multiplicity of the resonance): the invariant torus $\mathcal{T}_0$ is then foliated into invariant $d$-dimensional tori on which the dynamics is again minimal and uniquely ergodic. We will always assume that $1\leq d \leq n$, as the case $d=0$ is trivial since it corresponds to the zero vector and hence to an invariant torus which consists uniquely of equilibrium solutions. The special case $d=1$ will play a very important role in our approach: in this case, writing $\omega=v$ to distinguish it from the general case, $F=F_v$ is just the real line generated by $v$ so there exists $t>0$ such that $tv\in\Z^{n}\setminus\{0\}$. Letting $T>0$ be the infimum of the set of such $t$, the vector $v$ will be called $T$-periodic, and it is easy to see that the orbits of the linear flow with frequency $v$ are all periodic with minimal period $T$.  

Now, for a general vector $\omega \in \R^n \setminus\{0\}$, one can associate a constant $Q_\omega>0$ and a real-valued function $\Psi_\omega$ defined for all real numbers $Q\geq Q_\omega$, which is non-decreasing and unbounded,  by
\begin{equation}\label{psi}
\Psi_\omega(Q)=\max\left\{|k\cdot\omega|^{-1} \; | \; k\in \Z^n \cap F , \; 0<|k|\leq Q \right\}
\end{equation}
where $\cdot$ denotes the Euclidean scalar product and $|\,.\,|$ is the supremum norm for vectors. By definition, we have
\[ |k\cdot\omega|\geq \frac{1}{\Psi_\omega(Q)}, \quad k\in \Z^n \cap F, \; 0<|k|\leq Q. \]
Special classes of vectors are obtained by prescribing the growth of this function. An important class are the so-called Diophantine vectors: a vector $\omega$ is called Diophantine if there exist constants $\gamma>0$ and $\tau \geq d-1$ such that $\Psi_\omega(Q)\leq \gamma^{-1}Q^\tau$. We denote by $\Omega_d(\gamma,\tau)$ the set of such vectors. For $d=1$, recall that $\omega=v$ is $T$-periodic and it is easy to check that $\Psi_v(Q)\leq T$ and therefore any $T$-periodic vector belongs to $\Omega_1(T^{-1},0)$.

\paraga For $\varepsilon>0$, the dynamics of the perturbed system can be extremely complicated. Our aim here is to give some information on this dynamics in a neighbourhood of the unperturbed invariant torus $\mathcal{T}_0$. 

For an analytic Hamiltonian system, it is well-known that if the frequency vector is Diophantine, the system can be analytically conjugated to a simpler system where the perturbation has been split into two parts: a resonant part, which captures the important features of the system and whose size is still comparable to $\varepsilon$, and a non-resonant part, whose size can be made exponentially small with respect to $\varepsilon^{-a}$, where the exponent $a>0$ depends only on the Diophantine exponent $\tau$. The result can also be extended to an arbitrary vector $\omega\in \R^n$, in which case the non-resonant part is exponentially small with respect to some function of $\varepsilon^{-1}$, this function depending only on $\Psi_\omega$. Such simpler systems are usually called resonant formal forms with a small remainder, and, among other things, they are very important in trying to obtain stability estimates for the evolution of the action variables and ``splitting" estimates when the unperturbed invariant torus becomes ``hyperbolic" for $\varepsilon>0$.

In the first part of this work, we extend these results, which were valid for analytic Hamiltonians, to the broader class of Gevrey Hamiltonians, and for an arbitrary frequency vector $\omega$. To do this, following \cite{BN09} and \cite{BF12}, we introduced a method of periodic approximations that reduced the general case $1 \leq d \leq n$ to the periodic case $d=1$. 

Our aim in this second part is to treat finitely differentiable Hamiltonians. Of course, the exponential smallness of the remainder in the normal form we obtained in the analytic or in the Gevrey case will be replaced by a polynomial smallness. The method we will use in this second part is, in spirit, analogous to the one we used in the first part. However, the technical details are different so we need to give a complete proof, and, moreover, the method we used in the first part has to be modified in order to reach a precise polynomial estimate on the remainder in the normal form in terms of the regularity of the system. We will also give applications to stability and splitting estimates, but this will be completely analogous to the first part so we will omit the details. 

\paraga Let us now state precisely our results, starting with the regularity assumption.

Let $k\geq 2$ be an integer, and let us denote by $C^k(\mathcal{D}_R)$ the Banach space of functions on $\mathcal{D}_R$ of class $C^k$, with the norm
\[ |f|_{C^k(\mathcal{D}_R)}=\max_{l\in \N^{2n}, |l|\leq k}|\partial^l f|_{C^{0}(\mathcal{D}_R)}, \quad f \in C^k(\mathcal{D}_R). \]
Now given an integer $p \geq 1$, let us denote by $C^k(\mathcal{D}_R,\R^p)$ the Banach space of functions from $\mathcal{D}_R$ to $\R^p$ of class $C^k$, with the norm
\[ |F|_{C^k(\mathcal{D}_R,\R^p)}=\max_{1 \leq i \leq p}|f_i|_{C^k(\mathcal{D}_R)}, \quad F=(f_1,\dots,f_p) \in C^k(\mathcal{D}_R,\R^p). \]
For simplicity, we shall simply write $|\,.\,|_{k}=|\,.\,|_{C^k(\mathcal{D}_R)}$ and $|\,.\,|_{k}=|\,.\,|_{C^k(\mathcal{D}_R,\R^p)}$.  

We first consider a Hamiltonian $H \in C^{k}(\mathcal{D}_R)$ of the form
\begin{equation}\label{Hlin}
\begin{cases} \tag{C1}
H(\theta,I)=l_\omega(I)+f(\theta,I), \quad (\theta,I)\in \mathcal{D}_R, \\
|f|_{k} \leq \varepsilon, \quad k\geq 2. 
\end{cases}
\end{equation}
We denote by $\{\,.\,,\,.\,\}$ the Poisson bracket associated to the symplectic structure on $\mathcal{D}_R$. For any vector $w \in \R^n$, let $X_w^t$ be the Hamiltonian flow of the linear integrable Hamiltonian $l_w(I)=w\cdot I$, and given any $g\in C^{1}(\mathcal{D}_R)$, we define
\begin{equation}\label{ave}
[g]_w=\lim_{s\rightarrow +\infty}\frac{1}{s}\int_{0}^{s}g\circ X_{w}^{t}dt.
\end{equation}
Note that $\{g,l_w\}=0$ if and only if $g \circ X_w^t=g$ if and only if $g=[g]_w$. 

Recall that the function $\Psi_\omega$ has been defined in~\eqref{psi}, then we define the functions
\[ \Delta_\omega(Q)=Q\Psi_\omega(Q), \; Q \geq Q_\omega, \quad \Delta_{\omega}^*(x)=\sup\{Q \geq Q_\omega\; | \; \Delta_{\omega}(Q)\leq x\}, \; x \geq \Delta_\omega(Q_\omega). \]
Our first result is the following.

\begin{theorem}\label{thmlin}
Let $H$ be as in~\eqref{Hlin}, and $\kappa\in \N$ such that $0\leq \kappa \leq k-1$. There exist positive constants $c$, $c_1$, $c_2$, $c_3$ and $c_4$ that depend only on $n,R,\omega,k$ and $\kappa$ such that if 
\begin{equation}\label{thr1}
\Delta_{\omega}^*(c\varepsilon^{-1}) \geq c_1,
\end{equation}
then there exists a symplectic map $\Phi_\kappa \in C^{k-\kappa}(\mathcal{D}_{R/2},\mathcal{D}_R)$ such that
\[ H \circ \Phi_\kappa = l_\omega +[f]_\omega+ g_\kappa + f_\kappa, \quad \{g_\kappa,l_\omega\}=0  \]
with the estimates 
\begin{equation}\label{estlindist}
|\Phi_\kappa-\mathrm{Id}|_{k-\kappa} \leq c_2 \Delta_{\omega}^*(c\varepsilon^{-1})^{-1}
\end{equation} 
and
\begin{equation}\label{estlin}
|g_\kappa|_{k-\kappa+1}\leq c_3 \varepsilon\Delta_{\omega}^*(c\varepsilon^{-1})^{-1}, \quad |f_\kappa|_{k-\kappa}\leq c_4 \varepsilon (\Delta_{\omega}^*(c\varepsilon^{-1}))^{-\kappa}.
\end{equation}
\end{theorem}

For any $0 \leq \kappa \leq k-1$, the above theorem states the existence of a symplectic conjugacy of class $C^{k-\kappa}$, close to identity, between the original Hamiltonian and a Hamiltonian which is the sum of the integrable part, the average of the perturbation whose size is of order $\varepsilon$, a resonant part which by definition Poisson commutes with the integrable part and whose size is of order $\varepsilon(\Delta_{\omega}^*(c\varepsilon^{-1}))^{-1}$, and a general part whose size is now of order $\Delta_{\omega}^*(c\varepsilon^{-1})^{-\kappa}$. The first terms of this Hamiltonian, namely $l_\omega+[f]_\omega+g$, is what is called a resonant normal form, and the last term $\tilde{f}$ is a ``small" remainder. 

For $\kappa=0$, the statement is trivial, and for $\kappa \geq 1$, the statement follows by applying $\kappa$ times an averaging procedure. The crucial point is that, using our method of periodic approximations, this averaging procedure will be further decomposed into $d$ ``elementary" periodic averaging. The outcome is that the loss of differentiability is in some sense minimal, after $\kappa$ steps, we only loose $\kappa$ derivatives independently of the vector $\omega$ and the number of effective frequencies $d$ (using a classical averaging procedure, this loss of differentiability would strongly depend on $\omega$ and $d$).  

Now in the Diophantine case, the estimates of Theorem~\ref{thmlin} can be made more explicit. Indeed, we have the upper bound $\Psi_\omega (Q)\leq \gamma^{-1}Q^{\tau}$ which gives the lower bound $\Delta_{\omega}^*(c\varepsilon^{-1}) \geq (c\gamma\varepsilon^{-1})^{\frac{1}{1+\tau}}$. The following corollary is then immediate.

\begin{corollary}\label{thmlinD}
Let $H$ be as in~\eqref{Hlin}, with $\omega \in \Omega_d(\gamma,\tau)$ and $\kappa\in \N$ such that $0\leq \kappa \leq k-1$. There exist positive constants $c$, $c_1$, $c_2$, $c_3$ and $c_4$ that depend only on $n,R,\omega,\tau,k$ and $\kappa$ such that if 
\begin{equation*}
\varepsilon \leq cc_1^{-(1+\tau)}\gamma,
\end{equation*}
then there exists a symplectic map $\Phi_\kappa \in C^{k-\kappa}(\mathcal{D}_{R/2},\mathcal{D}_R)$ such that
\[ H \circ \Phi_\kappa = l_\omega +[f]_\omega+ g_\kappa + f_\kappa, \quad \{g_\kappa,l_\omega\}=0  \]
with the estimates 
\begin{equation*}
|\Phi_\kappa-\mathrm{Id}|_{k-\kappa} \leq c_2 (c^{-1}\gamma^{-1}\varepsilon)^{\frac{1}{1+\tau}}
\end{equation*} 
and
\begin{equation*}
|g_\kappa|_{k-\kappa+1}\leq c_3 \varepsilon(c^{-1}\gamma^{-1}\varepsilon)^{\frac{1}{1+\tau}}, \quad |f_\kappa|_{k-\kappa}\leq c_4 \varepsilon (c^{-1}\gamma^{-1}\varepsilon)^{\frac{\kappa}{1+\tau}}.
\end{equation*}
\end{corollary}

\paraga We can also consider a perturbation of non-linear integrable Hamiltonian, that is a Hamiltonian $H \in C^{k}(\mathcal{D}_R)$ of the form
\begin{equation}\label{Hnonlin}
\begin{cases} \tag{C2}
H(\theta,I)=h(I)+f(\theta,I), \quad (\theta,I)\in \mathcal{D}_R, \\
\nabla h(0)=\omega, \quad |h|_{k+2} \leq 1, \quad |f|_{k} \leq \varepsilon, \quad k\geq 2. 
\end{cases}
\end{equation} 
Note that here, for reasons that will be explained below, it is more convenient to assume that the integrable Hamiltonian is $C^{k+2}$ together with a control on its $C^{k+2}$ norm. 

For a ``small" parameter $r>0$, we will focus on the domain $\mathcal{D}_r=\T^n \times B_r$, which is a neighbourhood of size $r$ of the unperturbed torus $\mathcal{T}_0=\T^n \times \{0\}$. 

Since we are interested in $r$-dependent domains in the space of action, the estimates for the derivatives with respect to the actions will have different size than the one for the derivatives with respect to the angles. To distinguish between them, we will split multi-integers $l\in \N^{2n}$ as $l=(l_1,l_2)\in \N^n \times \N^n$ so that $\partial^l =\partial_\theta^{l_1}\partial_I^{l_2}$ and $|l|=|l_1|+|l_2|$. Let us denote by $\mathrm{Id}_I$ and $\mathrm{Id}_\theta$ the identity map in respectively the action and angle space, and for a function $F$ with values in $\T^n \times \R^n$, we shall write $F=(F_\theta,F_I)$.

\begin{theorem}\label{thmnonlin}
Let $H$ be as in~\eqref{Hnonlin} and $\kappa\in \N$ such that $0\leq \kappa \leq k-1$. There exist positive constants $c$, $c_5$, $c_6$, $c_7$ and $c_8$ that depend only on $n,R,\omega,k$ and $\kappa$, such that if 
\begin{equation}\label{thr2}
\sqrt{\varepsilon}\leq r \leq R, \quad \Delta_{\omega}^*(cr^{-1}) \geq c_5, 
\end{equation}
there exists a symplectic map $\Phi_\kappa \in C^{k-\kappa}(\mathcal{D}_{r/2},\mathcal{D}_{r})$, $\Phi_\kappa=(\Phi_{\kappa,\theta},\Phi_{\kappa,I})$, such that
\[ H \circ \Phi_\kappa = h + [f]_\omega+g_\kappa+ f_\kappa, \quad \{g_\kappa,l_\omega\}=0.  \]
Moreover, for $l=(l_1,l_2)\in \N^{2n}$, we have the estimates 
\[ |\partial^l g_\kappa|_{C^0(\mathcal{D}_{r/2})}\leq c_6 r^2r^{-|l_2|}\Delta_{\omega}^*(cr^{-1})^{-1}, \] 
for $|l|\leq k-\kappa+1$, and
\[ |\partial^l(\Phi_{\kappa,I}-\mathrm{Id}_I)|_{C^0(\mathcal{D}_{r/2})} \leq c_7 rr^{-|l_2|}\Delta_{\omega}^*(cr^{-1})^{-1},\] 
\[ |\partial^l(\Phi_{\kappa,\theta}-\mathrm{Id}_\theta)|_{C^0(\mathcal{D}_{r/2})} \leq c_7r^{-|l_2|}\Delta_{\omega}^*(cr^{-1})^{-1},\]
\[ |\partial^l f_\kappa|_{C^0(\mathcal{D}_{r/2})}\leq c_8 r^2r^{-|l_2|}\Delta_{\omega}^*(cr^{-1})^{-\kappa}\]
for $|l|\leq k-\kappa$.
\end{theorem}

The proof of Theorem~\ref{thmnonlin} is an easy consequence of Theorem~\ref{thmlin}: by a localization procedure, the Hamiltonian~\eqref{Hnonlin} on the domain $\mathcal{D}_r$ is, for $r\geq \sqrt{\varepsilon}$, equivalent to the Hamiltonian~\eqref{Hlin} on the domain $\mathcal{D}_1$, but with $r$ instead of $\varepsilon$ as the small parameter. This is completely analogous to \cite{BouI12}, so we will not repeat the argument. 

The only minor difference is that we assumed $|h|_{k+2}\leq 1$ so that when reducing the proof of Theorem~\ref{thmnonlin} to Theorem~\ref{thmlin}, we obtain a perturbation for which we can still control its $C^{k}$ norm. We could have assumed $|h|_{k}\leq 1$ in \eqref{Hnonlin}, but then we would have had a control only on the $C^{k-2}$ norm of the perturbation, and so the assumptions $k \geq 2$ and $1 \leq \kappa \leq k-1$ in the above statement should have been replaced by $k \geq 4$ and $1 \leq \kappa \leq k-3$.

Now as before, in the Diophantine case we can give a more concrete statement.

\begin{corollary}\label{thmnonlinD}
Let $H$ be as in~\eqref{Hnonlin}, with $\omega \in \Omega_d(\gamma,\tau)$ and $\kappa\in \N$ such that $0\leq \kappa \leq k-1$. There exist positive constants $c$, $c_5$, $c_6$, $c_7$ and $c_8$ that depend only on $n,R,\omega,k$ and $\kappa$, such that if 
\begin{equation}\label{thr3}
\sqrt{\varepsilon}\leq r \leq R, \quad r \leq cc_5^{-(1+\tau)}\gamma,
\end{equation}
there exists a symplectic map $\Phi_\kappa \in C^{k-\kappa}(\mathcal{D}_{r/2},\mathcal{D}_{r})$, $\Phi=(\Phi_{\kappa,\theta},\Phi_{\kappa,I})$, such that
\[ H \circ \Phi_\kappa = h + [f]_\omega+g_\kappa +f_\kappa, \quad \{g_\kappa,l_\omega\}=0.  \]
Moreover, for $l=(l_1,l_2)\in \N^{2n}$, we have the estimates 
\[ |\partial^l g_\kappa|_{C^0(\mathcal{D}_{r/2})}\leq c_6 r^2r^{-|l_2|}(c^{-1}\gamma^{-1}r)^{\frac{1}{1+\tau}}\]
for $|l|\leq k-\kappa+1$, and
\[ |\partial^l(\Phi_{\kappa,I}-\mathrm{Id}_I)|_{C^0(\mathcal{D}_{r/2})} \leq c_7 rr^{-|l_2|}(c^{-1}\gamma^{-1}r)^{\frac{1}{1+\tau}},\]
\[ |\partial^l(\Phi_{\kappa,\theta}-\mathrm{Id}_\theta)|_{C^0(\mathcal{D}_{r/2})} \leq c_7r^{-|l_2|}(c^{-1}\gamma^{-1}r)^{\frac{1}{1+\tau}},\]
\[ |\partial^l f_\kappa|_{C^0(\mathcal{D}_{r/2})}\leq c_8 r^2r^{-|l_2|}(c^{-1}\gamma^{-1}r)^{\frac{\kappa}{1+\tau}}\]
for $|l|\leq k-\kappa$.
\end{corollary}

\paraga Finally let us describe the plan of this paper. In \S\ref{s2}, we will deduce from our main results Theorem~\ref{thmlin} and Theorem~\ref{thmnonlin} polynomially large stability estimates for the evolution of the action variables, which are global for perturbations of linear integrable systems and only local for perturbations of non-linear integrable systems. In the first case, that is for perturbations of linear integrable systems, we will show on an example that these estimates are very accurate. We will also use Theorem~\ref{thmnonlin} to prove a result of polynomial smallness for the splitting of invariant manifolds of a hyperbolic tori. The proof of the main results will be given in \S\ref{s3}: as we already explained, it will be enough to prove Theorem~\ref{thmlin}, as Theorem~\ref{thmnonlin} follows from Theorem~\ref{thmlin} by a procedure we already described in details in the first part of this work. Finally, we will gather in a appendix some technical estimates concerning finitely differentiable functions that are used to prove our theorems.

To simplify the notations and improve the readability, when convenient we shall replace constants depending on $n,R,\omega,k$ and $\kappa$ that can, but need not be, made explicit by a $\cdot$, that is an expression of the form $u \MP v$  means that there exist a constant $c>0$, that depends on the above set of parameters, such that $u \leq cv$. Similarly, we will use the notations $u \PS v$ and $u \EP v$.

\section{Applications to stability and splitting estimates}\label{s2}

In this section, we will give consequences of our normal forms Theorem~\ref{thmlin} and Theorem~\ref{thmnonlin} to the stability of the action variables, and to the splitting of invariant manifolds of a hyperbolic tori. These applications are completely analogous to the ones contained in \cite{BouI12}, section \S 2 and \S 3, therefore we shall not repeat the details of the proofs.

\paraga Let us start by describing the stability estimates we can obtain for a perturbation of a linear integrable Hamiltonian as in~\eqref{Hlin}. Using Theorem~\ref{thmlin}, it is well-known that the action variables of the transformed Hamiltonian $H \circ \Phi_\kappa$ have only small variations in the direction given by the subspace $F=F_\omega$ during an interval of time which is as large as the inverse of the size of the remainder: this follows from the Hamiltonian form of the equations, and the fact that the Hamiltonians $[f]_\omega$ and $g$ Poisson commutes with the integrable part $l_\omega$. Moreover, using the fact that $\Phi_\kappa$ is symplectic and close to identity, the same property remains true for the Hamiltonian $H$.

Now the larger we take $\kappa$, the smaller is the size of the remainder and hence the longer is the stability time. The statement of Theorem~\ref{thmlin} allows us to take $\kappa$ as large as $k-1$, however, with this value of $\kappa$ the transformed Hamiltonian would be only $C^1$ and the existence of its Hamiltonian flow would become an issue. So we can only apply Theorem~\ref{thmlin} with $\kappa=k-2$, and this yields the following result.  

\begin{theorem}\label{stablin}
Under the notations and assumptions of Theorem~\ref{thmlin}, let $(\theta(t),I(t))$ be a solution of the Hamiltonian system defined by $H$, with $(\theta_0,I_0)\in \mathcal{D}_{R/4}$ and let $T_0$ be the smallest $t \in ]0,+\infty]$ such that $(\theta(t),I(t))\notin \mathcal{D}_{R/4}$. Then we have the estimates
\[ |\Pi_F(I(t)-I_0)| \leq \delta, \quad |t|\MP \min \left\{T_0, \delta\varepsilon^{-1}\Delta_\omega^*(\cdot\varepsilon^{-1})^{k-2}\right\} \]
for any $(\Delta_\omega^*(\cdot\varepsilon^{-1}))^{-1} \MP \delta\MP 1$. Moreover, if $F=\R^n$, then 
\[ |I(t)-I_0| \leq \delta, \quad |t|\MP \delta\varepsilon^{-1}\Delta_\omega^*(\cdot\varepsilon^{-1})^{k-2}. \]
\end{theorem}

\begin{corollary}\label{stablinD}
Under the notations and assumptions of Corollary~\ref{thmlinD}, let $(\theta(t),I(t))$ be a solution of the Hamiltonian system defined by $H$, with $(\theta_0,I_0)\in \mathcal{D}_{R/4}$ and let $T_0$ be the smallest $t \in ]0,+\infty]$ such that $(\theta(t),I(t))\notin \mathcal{D}_{R/4}$. Then we have the estimates
\[ |\Pi_F(I(t)-I_0)| \leq \delta, \quad |t|\MP \min \left\{T_0, \delta\varepsilon^{-1}(\gamma\varepsilon^{-1})^{\frac{k-2}{1+\tau}}\right\} \]
for any $(\gamma^{-1}\varepsilon)^{\frac{1}{1+\tau}} \MP \delta\MP 1$. Moreover, if $F=\R^n$, then 
\[ |I(t)-I_0| \leq \delta, \quad |t|\MP \delta\varepsilon^{-1}(\gamma\varepsilon^{-1})^{\frac{k-2}{1+\tau}}. \]
\end{corollary}

Note that in the particular case where $\delta \EP (\Delta_\omega^*(\cdot\varepsilon^{-1}))^{-1}$, we have 
\[ |\Pi_F(I(t)-I_0)| \MP (\Delta_\omega^*(\cdot\varepsilon^{-1}))^{-1}\]
for
\[|t|\MP \min\left\{T_0,\varepsilon^{-1}\Delta_\omega^*(\cdot\varepsilon^{-1})^{k-3}\right\} \]
and in the Diophantine case, for $\delta \EP (\gamma^{-1}\varepsilon)^{\frac{1}{1+\tau}}$, we have
\[ |\Pi_F(I(t)-I_0)| \MP (\gamma^{-1}\varepsilon)^{\frac{1}{1+\tau}}\]
for
\[|t|\MP \min\left\{T_0,\cdot\varepsilon^{-1}(\gamma\varepsilon^{-1})^{\frac{k-3}{1+\tau}}\right\}. \]

\paraga Next, we will show on an example that the estimates of Theorem~\ref{stablin}, and therefore the estimate on the remainder of Theorem~\ref{thmlin}, are almost ``essentially" optimal. 

\begin{theorem}\label{explin}
Let $\omega\in \R^n\setminus\{0\}$. Then there exist a sequence $(\varepsilon_j)_{j\in\N}$ of positive real numbers and a sequence $(f_j)_{j\in\N}$ of functions in $C^{\infty}(\mathcal{D}_R)$, with 
\[ \varepsilon_j \MP |f_j|_{k} \MP \varepsilon_j, \quad \lim_{j \rightarrow +\infty}\varepsilon_j=0,\] 
such that for $j \PS 1$, the Hamiltonian system defined by $H_j=l_\omega+f_j$ has solutions $(\theta(t),I(t))$ which satisfy
\[ |t|\varepsilon_j (\Delta_\omega^*(\cdot\varepsilon_j^{-1}))^{-(k-1)} \MP |\Pi_F(I(t)-I_0)| \MP |t|\varepsilon_j (\Delta_\omega^*(\cdot\varepsilon_j^{-1}))^{-(k-1)}. \]
\end{theorem}

The example is exactly the same as in the corresponding statement in \cite{BouI12}, the only difference being, of course, that we estimate here its $C^k$ norm instead of its Gevrey norm.

It is easy to observe that there is a little discrepancy between Theorem~\ref{stablin} and Theorem~\ref{explin}: in Theorem~\ref{stablin} we find the exponent $k-2$, while in Theorem~\ref{explin} the exponent is $k-1$. This can be explained as follows. We already know why we could not reach the exponent $k-1$ in Theorem~\ref{stablin} in general, but if the Hamiltonian is not just $C^k$ and $C^k$-close to the integrable but $C^{\infty}$ (in fact, $C^{k+1}$ would be enough) and $C^k$-close to the integrable, we could actually reach this exponent $k-1$ in Theorem~\ref{stablin}. Now the Hamiltonian which appears in Theorem~\ref{explin} is indeed smooth (it is in fact analytic) and $C^k$-close to integrable, and this explains why the exponent is $k-1$.

\paraga We also have stability estimates for a perturbation of a non-linear integrable Hamiltonian as in~\eqref{Hnonlin}, but unfortunately we do not know how to construct an example to see if these estimates are very accurate.

\begin{theorem}\label{stabnonlin}
Under the notations and assumptions of Theorem~\ref{thmnonlin}, let $(\theta(t),I(t))$ be a solution of the Hamiltonian system defined by $H$, with $(\theta_0,I_0)\in \mathcal{D}_{r/4}$ and let $T_0$ be the smallest $t \in ]0,+\infty]$ such that $(\theta(t),I(t))\notin \mathcal{D}_{r/4}$. Then we have the estimates
\[ |\Pi_F(I(t)-I_0)| \leq \delta, \quad |t|\MP \min \left\{T_0, \delta r^{-2}\Delta_\omega^*(\cdot r^{-1})^{k-2}\right\} \]
for any $r(\Delta_\omega^*(\cdot r^{-1}))^{-1} \MP \delta\MP r$. Moreover, if $F=\R^n$, then 
\[ |I(t)-I_0| \leq \delta, \quad |t|\MP \delta r^{-2}\Delta_\omega^*(\cdot r^{-1})^{k-2}. \]
\end{theorem}

\begin{corollary}\label{stabnonlinD}
Under the notations and assumptions of Corollary~\ref{thmnonlinD}, let $(\theta(t),I(t))$ be a solution of the Hamiltonian system defined by $H$, with $(\theta_0,I_0)\in \mathcal{D}_{r/4}$ and let $T_0$ be the smallest $t \in ]0,+\infty]$ such that $(\theta(t),I(t))\notin \mathcal{D}_{r/4}$. Then we have the estimates
\[ |\Pi_F(I(t)-I_0)| \MP \delta, \quad |t|\MP \min \left\{T_0, \delta r^{-2}(\gamma r^{-1})^{\frac{k-2}{1+\tau}}\right\} \]
for any $r(\gamma^{-1}r)^{\frac{1}{1+\tau}} \MP \delta\MP r$. Moreover, if $F=\R^n$, then 
\[ |I(t)-I_0| \MP \delta, \quad |t|\MP \delta r^{-2}(\gamma r^{-1})^{\frac{k-2}{1+\tau}}. \]
\end{corollary}

\paraga Finally, we address the problem of estimating the ``splitting" of invariant manifolds provided a Hamiltonian system as in~\ref{Hnonlin} has a suitable ``hyperbolic" invariant torus of dimension $d$, where $1 \leq d \leq n-1$, coming from the unperturbed invariant torus $\mathcal{T}_0$. Such a torus will have stable and unstable manifolds which are Lagrangian, and assuming that these manifolds intersect, our aim is to estimate their ``angle" at the intersection point. Using the Lagrangian character of the invariant manifolds, as in \cite{LMS03} we can define a symmetric matrix of size $n$, called the splitting matrix, whose eigenvalues are called splitting angles, and our main result (Theorem~\ref{thmsplit} below) states that at least $d$ splitting angles are polynomially small.

Without loss of generality, we may already assume that our frequency vector $\omega$ is of the form $\omega=(\varpi,0)\in \R^d\times \R^m=\R^n$, with $\varpi \in \R^d$ non-resonant. Then we split our angle-action coordinates $(\theta,I) \in \mathcal{D}_R=\T^n \times B_R$ accordingly to  $\omega=(\varpi,0)\in \R^d\times \R^m=\R^n$: we write $\theta=(\theta_1,\theta_2) \in \T^d \times \T^m$ and $I=(I_1,I_2) \in B_R^d \times B_R^m$, where $B_R^d=B_R \cap \R^d$ and $B_R^m=B_R \cap \R^m$.

We will first consider an abstract Hamiltonian $\mathcal{H}=\mathcal{H}_{\lambda,\mu} \in C^2(\mathcal{D}_1)$, depending on two parameters $\lambda>0$ and $\mu>0$, of the form
\begin{equation}\label{Hres}
\begin{cases} 
\mathcal{H}_{\lambda,\mu}(\theta,I)=\mathcal{H}_{\lambda}(\theta_2,I)+\mu F(\theta,I), \quad |F|_{2}\MP 1  \\
\mathcal{H}_{\lambda}(\theta_2,I)=\mathcal{H}^{av}(\theta_2,I)+\lambda R(\theta_2,I), \quad |R|_{2} \MP 1\\
\mathcal{H}^{av}(\theta_2,I)=\sqrt{\varepsilon}^{-1}\varpi\cdot I_1+AI_1\cdot I_1+BI_2\cdot I_2+V(\theta_2), \quad |V|_{2} \leq 1.
\end{cases}
\end{equation}   
This Hamiltonian $\mathcal{H}_{\lambda,\mu}$ has to be considered as an arbitrary $\mu$-perturbation of $\mathcal{H}_{\lambda}=\mathcal{H}_{\lambda,0}$, and $\mathcal{H}_{\lambda}$ as a special $\lambda$-perturbation of the ``averaged" system $\mathcal{H}^{av}$ (special because the perturbation $R$ is independent of $\theta_1$). Note that the averaged system can be further decomposed as a sum of two Hamiltonians
\[ \mathcal{H}^{av}(\theta_2,I)=K(I_1)+P(\theta_2,I_2),   \]
where $K(I_1)=\sqrt{\varepsilon}^{-1}\varpi\cdot I_1+AI_1\cdot I_1$ is a completely integrable system on $\mathcal{D}_1^d=\T^d \times B_1^d$ and $P(\theta_2,I_2)=BI_2\cdot I_2+V(\theta_2) $ is a mechanical system (or a ``multidimensional pendulum") on $\mathcal{D}_1^m=\T^m \times B_1^m$.

Our first assumption is the following one:

\medskip

(A.1) The matrix $B$ is positive definite (or negative definite), and the function $V : \T^m \rightarrow \R$ has a non-degenerated maximum (or minimum).

\medskip

Without loss of generality, we may assume that $V$ reaches its maximum at $\theta_2=0$, so that $O=(0,0)\in \mathcal{D}_1^m$ is a hyperbolic fixed point for the Hamiltonian flow generated by $P$. This in turns implies that the set $\mathcal{T}=\{I_1=0\} \times O$ is a $d$-dimensional torus invariant for the averaged system, which is quasi-periodic with frequency $\sqrt{\varepsilon}^{-1}\varpi$, and it is hyperbolic in the sense that it has $C^1$ stable and unstable manifolds
\[ W^{\pm}(\mathcal{T})=\{I_1=0\} \times W^{\pm}(O) \]
where $W^{\pm}(O)$ are the stable and unstable manifolds of $O$, which are Lagrangian. 

Now this picture is easily seen to persist if we move from $\mathcal{H}^{av}$ to $\mathcal{H}_{\lambda}$. Indeed, since $\mathcal{H}_{\lambda}$ is still independent of $\theta_1$, the level sets of $I_1$ are still invariant, hence the Hamiltonian flow generated by the restriction of $\mathcal{H}_{\lambda}$ to $\{I_1=0\} \times \mathcal{D}_1^m$ (considered as a flow on $\mathcal{D}_1^m$) is a $\lambda$-perturbation (in the $C^1$-topology) of the Hamiltonian flow generated by $P$: as a consequence, it has a hyperbolic fixed point $O_\lambda \in \mathcal{D}_1^m$ which is $\lambda$-close to $O$, for $\lambda$ small enough. Hence $\mathcal{T}_\lambda=\{I_1=0\} \times O_\lambda$ is invariant under the Hamiltonian flow of $\mathcal{H}_{\lambda}$, quasi-periodic with frequency $\sqrt{\varepsilon}^{-1}\varpi$ and it is hyperbolic with Lagrangian stable and unstable manifolds
\[ W^{\pm}(\mathcal{T}_\lambda)=\{I_1=0\} \times W^{\pm}(O_\lambda). \] 

Our second assumption concerns the persistence of the torus $\mathcal{T}_\lambda$, as well as its stable and unstable manifolds, when we move from $\mathcal{H}_{\lambda}$ to $\mathcal{H}_{\lambda,\mu}$:  

\medskip

(A.2) For any $0 \leq \lambda \MP 1$ and $0 \leq \mu \MP \lambda$, the system $\mathcal{H}_{\lambda,\mu}$ has an invariant torus $\mathcal{T}_{\lambda,\mu}$, with $\mathcal{T}_{\lambda,0}=\mathcal{T}_{\lambda}$, of frequency $\sqrt{\varepsilon}^{-1}\varpi$, with $C^1$ stable and unstable manifolds $W^{\pm}(\mathcal{T}_{\lambda,\mu})$ which are exact Lagrangian graphs over fixed relatively compact domains $U^{\pm} \subseteq \T^n$. Moreover, $W^{\pm}(\mathcal{T}_{\lambda,\mu})$ are $\mu$-close to $W^{\pm}(\mathcal{T}_{\lambda})$ for the $C^{1}$-topology.

\medskip

Let us denote by $S^{\pm}_{\lambda,\mu}$ generating functions for $W^{\pm}(\mathcal{T}_{\lambda,\mu})$ over $U^{\pm}$, that is if $V^{\pm}=U^{\pm} \times B_1$, then 
\[ W^{\pm}(\mathcal{T}_{\lambda,\mu}) \cap V^{\pm}=\{ (\theta_*,I_*) \in V^{\pm} \; | \: I_*=\partial_\theta S^{\pm}_{\lambda,\mu}(\theta_*)  \} \]
where $S^{\pm}_{\lambda,\mu} : U^{\pm} \rightarrow \R$ are $C^2$ functions. Since  $W^{\pm}(\mathcal{T}_{\lambda,\mu})$ are $\mu$-close to $W^{\pm}(\mathcal{T}_{\lambda})$ for the $C^{1}$-topology, the first derivatives of the functions $S^{\pm}_{\lambda,\mu}$ are $\mu$-close to the first derivatives of $S^{\pm}_{\lambda}=S^{\pm}_{\lambda,0}$ for the $C^1$-topology.

Finally, our last assumption concerns the existence of orbits which are homoclinic to $\mathcal{T}_{\lambda,\mu}$:

\medskip

(A.3) For any $0 \leq \lambda \MP 1$ and $0 \leq \mu \MP \lambda$, the set $W^+(\mathcal{T}_{\lambda,\mu}) \cap W^-(\mathcal{T}_{\lambda,\mu}) \setminus \mathcal{T}_{\lambda,\mu}$ is non-empty. 

\medskip

Let $\gamma_{\lambda,\mu}$ be an orbit in $W^+(\mathcal{T}_{\lambda,\mu}) \cap W^-(\mathcal{T}_{\lambda,\mu}) \setminus \mathcal{T}_{\lambda,\mu}$, and $p_{\lambda,\mu}=\gamma_{\lambda,\mu}(0)=(\theta_{\lambda,\mu},I_{\lambda,\mu})$. Since $p_{\lambda,\mu}$ is a homoclinic point, $\theta_{\lambda,\mu} \in U^+ \cap U^-$ and $\partial_\theta S^{+}_{\lambda,\mu}(\theta_{\lambda,\mu})=\partial_\theta S^{-}_{\lambda,\mu}(\theta_{\lambda,\mu})$. Then we can define the splitting matrix $M(\mathcal{T}_{\lambda,\mu},p_{\lambda,\mu})$ of $\mathcal{T}_{\lambda,\mu}$ at the point $p_{\lambda,\mu}$, as the symmetric square matrix of size $n$
\[ M(\mathcal{T}_{\lambda,\mu},p_{\lambda,\mu})= \partial_\theta^2 (S_{\lambda,\mu}^+-S_{\lambda,\mu}^-)(\theta_{\lambda,\mu}). \]
Moreover, for $1\leq i \leq n$, we define the splitting angles $a_i(\mathcal{T}_{\lambda,\mu},p_{\lambda,\mu})$ as the eigenvalues of the matrix $M(\mathcal{T}_{\lambda,\mu},p_{\lambda,\mu})$.

Now exactly as in \cite{LMS03} for the analytic case or \cite{BouI12} for the Gevrey case, we have the following result.

\begin{theorem}\label{thmsplit0}
Let $\mathcal{H}_{\lambda,\mu}$ be as in~\eqref{Hres}, and assume that (A.1), (A.2) and (A.3) are satisfied. Then, with the previous notations, we have the estimates
\[ |a_i(\mathcal{T}_{\lambda,\mu},p_{\lambda,\mu})| \MP \mu, \quad 1 \leq i \leq d.\]
\end{theorem}

Now let us come back to a Hamiltonian system as in~\eqref{Hnonlin}, and we first make a simplifying assumption on the integrable part $h$:

\medskip

(A.4) The quadratic part of $h$ at $0 \in B_R$ can be written as
\[ \left( \begin{array}{cc}
A & 0  \\
0 & B  \\
\end{array} \right)\] 
where $A$ and $B$ are square matrix of size respectively $d$ and $m$.

\medskip

Under this assumption, it is easy to see that the Hamiltonian $H \circ \Phi_{k-2}$, given by Theorem~\ref{thmnonlin} with the value $r=2\sqrt{\varepsilon}$, can be written as in~\eqref{Hres}, with
\begin{equation}\label{mu}
\lambda=(\Delta_{\omega}^*(\cdot\sqrt{\varepsilon}^{-1}))^{-1}, \quad \mu=\lambda^{k-2}=(\Delta_{\omega}^*(\cdot\sqrt{\varepsilon}^{-1}))^{-k+2}, 
\end{equation}
provided that we rescale the time by $\sqrt{\varepsilon}$, and provided that $k \geq 3$ (we need $k\geq 3$ to insure that we have a control on the $C^2$ norm of the Hamiltonian $R$ in~\eqref{Hres}). We refer to \cite{BouI12} for more details.

Now as the solutions of the Hamiltonian system defined by $H \circ \Phi_{k-2}$ differs from those of $\mathcal{H}_{\lambda,\mu}$ (with $\lambda$ and $\mu$ as in~\eqref{mu}) only by a time change, $\mathcal{T}_{\lambda,\mu}$ is still an invariant hyperbolic torus for $H \circ \Phi$, with the same stable and unstable manifolds. Coming back to our original system, the torus $T_{\varepsilon}=\Phi(\mathcal{T}_{\lambda,\mu})$ is hyperbolic for $H$, with stable and unstable manifolds $W^{\pm}(T_{\varepsilon})=\Phi(W^{\pm}(\mathcal{T}_{\lambda,\mu}))$, and for $\gamma_{\varepsilon}=\Phi(\gamma_{\lambda,\mu})$ and $p_{\varepsilon}=\Phi(p_{\lambda,\mu})$ we can define a splitting matrix $M(\mathcal{T}_{\varepsilon},p_{\varepsilon})$ and splitting angles $a_i(\mathcal{T}_{\varepsilon},p_{\varepsilon})$ for $1\leq i \leq n$.

\begin{theorem}\label{thmsplit}
Let $H$ be as in~\eqref{Hnonlin}, with $r=2\sqrt{\varepsilon}$ satisfying~\eqref{thr2} and $k \geq 3$. Assume that (A.4) is satisfied, and that (A.1), (A.2) and (A.3) are satisfied for the Hamiltonian $\mathcal{H}_{\lambda,\mu}$ with $\lambda$ and $\mu$ as in~\eqref{mu}. Then, with the previous notations, we have the estimates
\[ |a_i(\mathcal{T}_{\varepsilon},p_{\varepsilon})| \MP \sqrt{\varepsilon}\left(1+(\Delta_{\omega}^*(\cdot\sqrt{\varepsilon}^{-1}))^{-1}\right) (\Delta_{\omega}^*(\cdot\sqrt{\varepsilon}^{-1}))^{-k+2}, \quad 1 \leq i \leq d.\]
\end{theorem}

The proof is analogous to \cite{BouI12}. In the Diophantine case, we have the following obvious corollary.

\begin{corollary}\label{thmsplitD}
Let $H$ be as in~\eqref{Hnonlin}, with $\omega \in \Omega_d(\gamma,\tau)$, $r=2\sqrt{\varepsilon}$ satisfying~\eqref{thr3} and $k \geq 3$. Assume that (A.4) is satisfied, and that (A.1), (A.2) and (A.3) are satisfied for the Hamiltonian $\mathcal{H}_{\lambda,\mu}$ with $\lambda$ and $\mu$ as in~\eqref{mu}. Then, with the previous notations, we have the estimates
\[ |a_i(\mathcal{T}_{\varepsilon},p_{\varepsilon})| \MP \sqrt{\varepsilon}\left(1+(\cdot\gamma^{-2}\varepsilon)^{\frac{1}{2(1+\tau)}}\right) (\cdot\gamma^{-2}\varepsilon)^{\frac{k-2}{2(1+\tau)}}, \quad 1 \leq i \leq d.\]
\end{corollary}

\section{Proof of the main results}\label{s3}

This section is devoted to the proof of Theorem~\ref{thmlin} (recall that Theorem~\ref{thmnonlin} follows from it exactly as in the first part of this work). The method is in principle analogous to the one used in \cite{BouI12}, but the technicalities are quite different so we need to give complete details.

\paraga The proof of Theorem~\ref{thmlin} starts with the special case $d=1$, that is when $F$ is one-dimensional, and $\kappa=1$. As we already said, in this situation the vector is in fact periodic so we shall denote it by $v$, and for any non-zero integer vector $k \in F$, we have the lower bound $|k\cdot v| \geq T^{-1}$ and so we will have $\Delta_v^*(\varepsilon^{-1})\geq(T\varepsilon)^{-1}$ in the statement of Theorem~\ref{thmlin} (or $\tau=0$ and $\gamma=T^{-1}$ in the statement of Corollary~\ref{thmlinD}) for this particular case.

Theorem~\ref{thmlin} in the case $d=1$ and for any $0\leq \kappa \leq k-1$ is essentially contained in \cite{Bou10}. However, in order to use the statement for the general case $1\leq d \leq n$, we will need a somewhat more general version.

So we introduce another parameter $\nu>0$ and we consider the Hamiltonian
\begin{equation}\label{Hper}
\begin{cases} 
H(\theta,I)=l_v(I)+s(\theta,I)+u(\theta,I), \quad (\theta,I)\in \mathcal{D}_R, \\
Tv \in \Z^n, \quad |s|_{i} \leq \nu, \quad |u|_{i} \leq \varepsilon, \quad 2 \leq i \leq k. 
\end{cases}
\end{equation}
Let us define $\delta=(2d(k-1))^{-1}R$. Then we have the following result.

\begin{proposition}\label{ana}
Let $H$ be as in \eqref{Hper}. Assume that
\begin{equation}\label{condana}
T\nu \MP 1, \quad \varepsilon \MP \nu.
\end{equation}
Then, for $R'=R-\delta$, there exists a symplectic map $\Theta \in C^{i-1}(\mathcal{D}_{R'},\mathcal{D}_{R})$ such that
\[ H\circ\Theta=l_v+s+[u]_v+u' \]
with the estimates
\[ |\Theta-\mathrm{Id}|_{i-1} \MP T\varepsilon, \quad |u'|_{i-1} \MP \varepsilon T\nu.  \]
\end{proposition}

Note that for the Hamiltonian~\eqref{Hper}, we consider $l_v$ as the unperturbed part, and $s+u$ as the perturbation. Since we have assumed that $\varepsilon\MP\mu$, the size of the perturbation is of order $\mu$, but as we will not alter the term $s$, the size of the ``effective" part of the perturbation is of order $\varepsilon$. The implicit constants in the above statement depend only on $n,R$ and $i$, but $i$ will eventually depend only on $k$ and $\kappa$. 

If we are only interested in the periodic case, then one may take $s=0$ in~\eqref{Hper}, $\varepsilon=\nu$ and write $u=f$ in the statement of Proposition~\ref{ana}, and this gives exactly the statement of Theorem~\ref{thmlin} in the case $\kappa=1$ if $i=k$ (with $\Phi_1=\Theta$, $g_1=0$ and $f_1=u'$). Then Theorem~\ref{thmlin} for $d=1$ and any $1 \leq \kappa \leq k$ follows by induction (the case $\kappa=0$ is of course trivial), but we won't need to prove this now as this will be proved later for an arbitrary $1 \leq d \leq n$.

\begin{proof}[Proof of Proposition~\ref{ana}]
Let us define
\[ \chi=T\int_{0}^{1}(u-[u]_v)\circ X^t_{Tv}tdt. \]    
The function $\chi$ is of class $C^{i}$, and obviously 
\[ |\chi_\kappa|_{i} \leq T|u|_{i} \leq T\varepsilon.\] 
Now let us denote by $X$ the Hamiltonian vector field associated to $\chi$, and by $X^t$ the time-$t$ map of the flow generated by $X$. Then $X$ is of class $C^{i-1}$, and obviously $|X|_{i-1} \leq |\chi|_{i}$. Moreover, by classical results on the existence of solutions of differential equations, $X^t$ is well-defined and of class $C^{i-1}$, for $|t|\leq \tau$ where $\tau \EP (|X|_{i-1})^{-1}$. By~\eqref{condana}, we may assume that $\tau>1$ and that $X^t$ sends $\mathcal{D}_{R'}$ into $\mathcal{D}_{R}$ for $|t|\leq 1$.

Now we can write
\begin{equation}\label{tay}
H \circ X^1=l_v\circ X^1+(s+u)\circ X^1  
\end{equation}
and using the general equality
\[ \frac{d}{dt}(H \circ X^t)=\{H,\chi\}\circ X^t, \]
we can apply Taylor's formula with integral remainder to the right-hand side of~\eqref{tay}, at order two for the first term and at order one for the second term, and we get
\begin{equation*}
H \circ X^1 = l_v+\{l_v,\chi\}+\int_{0}^{1}(1-t)\{\{l_v,\chi\},\chi\} \circ X^t dt + s+u+\int_{0}^{1}\{s+u,\chi\} \circ X^tdt. 
\end{equation*}
Now let us check that the equality $\{\chi,l_v\}=u-[u]_v$ holds true: we have
\[ \{\chi,l_v\}=v\cdot \partial_\theta \chi=T\int_{0}^{1}v\cdot\partial_\theta((u-[u]_v)\circ X_{Tv}^{t}) tdt=\int_{0}^{1} Tv\cdot\partial_\theta((u-[u]_v)\circ X_{Tv}^{t}) tdt \]
so using the chain rule
\[ \{\chi,l_v\}=\int_{0}^{1}\frac{d}{dt}((u-[u]_v)\circ X_{Tv}^{t}) tdt \]
and an integration by parts
\[ \{\chi,l_v\}=\left.((u-[u]_v)\circ X_{Tv}^{t})t\right\vert_{0}^{1}-\int_{0}^{1}(u-[u]_v)\circ X_{Tv}^{t}dt=u-[u]_v,  \]
where in the last equality, $(u-[u]_v)\circ X_{Tv}^{1}=u-[u]_v$ since $Tv\in\Z^n$ and the integral vanishes since it is easy to check that this integral equals $[u-[u]_v]_v=[u]_v-[u]_v=0$. So using the equality $\{\chi,l_v\}=u-[u]_v$, that can be written as $\{l_v,\chi\}+u=[u]_v$, we have 
\[ H \circ X_\kappa^1=l_v+s+[u]_v+\int_{0}^{1}(1-t)\{\{l_v,\chi\},\chi\} \circ X^t dt+\int_{0}^{1}\{s+u,\chi_\kappa\} \circ X^tdt,\]
and if we set
\[ u_{t}=tu+(1-t)[u]_v, \quad u'=\int_{0}^{1}\{u_{t}+s,\chi\} \circ X^tdt \]
and use again the equality $\{\chi,l_v\}=u-[u]_v$ we eventually obtain
\[ H \circ X^1=l_v+s+[u]_v+u'.\] 
So we define $\Theta=X^1$. 

Now let us check the estimates. First, by~\eqref{condana} we can apply Lemma~\ref{tech} to obtain
\[ |\Theta-\mathrm{Id}|_{i-1} \MP |X|_{i-1} \MP T\varepsilon.  \]
Then, we can estimate
\[ |u'|_{i-1} \leq |\{u_{t}+s,\chi\} \circ X^t|_{i-1} \]
and using the estimate~\eqref{faa1} from Appendix~\ref{app1}, this gives
\[ |u'|_{i-1} \MP |\{u_{t}+s,\chi\} |_{i-1}. \]
Then, using the estimate~\eqref{poisson} from Appendix~\ref{app1},
\[ |u'|_{i-1} \MP (|u_{t}|_{i}+|s|_{i})|\chi|_{i} \]
which finally gives 
\[ |u'|_{i-1} \MP (\varepsilon+\nu)T\varepsilon \MP \nu T\varepsilon.\]
This proves the proposition.
\end{proof}

\paraga Now we recall a result from Diophantine approximation that will be crucial to go from the special case $d=1$ to the general case $1 \leq d \leq n$.

\begin{proposition}\label{dio}
Let $\omega \in \R^n\setminus\{0\}$. For any $Q\PS 1$, there exist $d$ periodic vectors $v_1, \dots, v_d$, of periods $T_1, \dots, T_d$, such that $T_1v_1, \dots, T_dv_d$ form a $\Z$-basis of $\Z^n \cap F$ and for $j\in\{1,\dots,d\}$,
\[ |\omega-v_j|\MP(T_j Q)^{-1}, \quad 1 \MP T_j \MP \Psi_\omega(Q).\]
\end{proposition}

For the proof, we refer to \cite{BF12}, Proposition $2.3$. The implicit constants depend only on $d$ and $\omega$. Note that the proposition is trivial for $d=1$, that is when $F_\omega$ is one dimensional, as $\omega=v$ is then periodic and it is sufficient to let $v_1=v$.

Now a consequence of the fact that the vectors $T_1v_1, \dots, T_dv_d$ form a $\Z$-basis of $\Z^n \cap F$ is contained in the following corollary. For simplicity, we shall write $[\,\cdot\,]_{v_1,\dots,v_d}=[\cdots[\,\cdot\,]_{v_1}\cdots]_{v_d}$, where $[\,\cdot\,]_w$ has been defined for an arbitrary vector $w$ in~\eqref{ave}.

\begin{corollary}\label{cordio}
Under the assumptions of Proposition~\ref{dio}, let $l_\omega(I)=\omega\cdot I$ and $l_{v_j}(I)=v_j\cdot I$ for $j\in\{1,\dots,d\}$. For any $g\in C^{\infty}(\mathcal{D}_R)$, we have $[g]_\omega=[g]_{v_1,\dots,v_d}$ and therefore $\{g,l_\omega\}=0$ if and only if $\{g,l_{v_j}\}=0$ for any $j\in\{1,\dots,d\}$.
\end{corollary} 

For a proof, we refer to \cite{Bou12}, Corollary $4.2$.

\paraga Now we can finally prove Theorem~\ref{thmlin}, which follows easily from the proposition below.

\begin{proposition}\label{proplin}
Let $H$ be as in \eqref{Hlin}, and $\kappa\in \N$ such that $0\leq \kappa \leq k-1$. For $Q\geq 1$, assume that 
\begin{equation}\label{condlin}
Q \PS 1, \quad \varepsilon \MP \Delta_\omega(Q)^{-1},
\end{equation}
then, for $R_{d\kappa}=R-d\kappa\delta$, there exists a symplectic map $\Phi_\kappa \in C^{k-\kappa}(\mathcal{D}_{R_{d\kappa}},\mathcal{D}_R)$ such that 
\[ H \circ \Phi_\kappa = l_\omega + [f]_\omega+g_\kappa + f_\kappa, \quad \{g_\kappa,l_\omega\}=0  \]
with the estimates
\[ |\Phi_\kappa-\mathrm{Id}|_{k-\kappa} \MP Q^{-1}, \quad |g_\kappa|_{k-\kappa+1}\MP \varepsilon Q^{-1}, \quad |f_\kappa|_{k-\kappa}\MP \varepsilon Q^{-\kappa}. \]
\end{proposition}

\begin{proof}[Proof of Theorem~\ref{thmlin}]
We choose 
\[ Q = \Delta_\omega^*(\cdot\varepsilon^{-1}) \]
with a well-chosen implicit constant so that the second part of~\eqref{condlin} is satisfied. Proposition~\ref{proplin} with this value of $Q$ implies Theorem~\ref{thmlin}, as the first part of~\eqref{condlin} is satisfied by the threshold~\eqref{thr1} and $R_{d\kappa}\geq R/2$ for any $0 \leq \kappa \leq k-1$.
\end{proof}

So it remains to prove Proposition~\ref{proplin}. Note that here we have to modify the approach taken in \cite{BouI12}. Indeed, if we follow \cite{BouI12}, Proposition~\ref{ana} would be needed for an arbitrary $0 \leq \kappa \leq k-1$ (that is, we would need Proposition~\ref{proplin} in the special case where $d=1$), and then Proposition~\ref{proplin} for a given $0 \leq \kappa \leq k-1$ would be obtained from Proposition~\ref{ana} for the corresponding $\kappa$, by an induction on $1\leq j \leq d$, noticing that the case $j=1$ corresponds to Proposition~\ref{ana} and that we are interested in the case $j=d$. This works, but the estimates in Proposition~\ref{proplin} are worse as the loss of differentiability depends on $d$: instead of loosing $\kappa$ derivatives, we would loose $d\kappa$ derivatives, that is the estimates would apply to the $C^{k-d\kappa}$ norms instead of the $C^{k-\kappa}$ norms.

To overcome this, we will have to make a double induction. We will prove Proposition~\ref{proplin} by induction on $0 \leq \kappa \leq k-1$, starting with the fact that the statement is trivial for $\kappa=0$. Then, to prove the inductive step, that is to go from $\kappa$ to $\kappa+1$, we will make an induction on $1 \leq j \leq d$ and use the statement of Proposition~\ref{ana} (which is Proposition~\ref{proplin} for $d=1$ but only $\kappa=1$). This leads to a slightly more complicated proof, but eventually leads to better estimates.

\begin{proof}[Proof of Proposition~\ref{proplin}]
The proof goes by induction on $0 \leq \kappa \leq k-1$. For $\kappa=0$, the statement is obviously true if we let $\Phi_0$ be the identity, $g_0=0$ and $f_0=f-[f]_\omega$. So now we assume that the statement holds true for some $0 \leq \kappa \leq k-2$, and we need to show that it remains true for $\kappa+1$.

By the induction hypothesis, there exists a symplectic map $\Phi_\kappa \in C^{k-\kappa}(\mathcal{D}_{R_{d\kappa}},\mathcal{D}_R)$ such that 
\[ H_\kappa=H \circ \Phi_\kappa = l_\omega + [f]_\omega+g_\kappa + f_\kappa, \quad \{g_\kappa,l_\omega\}=0  \]
with the estimates
\[ |\Phi_\kappa-\mathrm{Id}|_{k-\kappa} \MP Q^{-1}, \quad |g_\kappa|_{k-\kappa+1}\MP \varepsilon Q^{-1}, \quad |f_\kappa|_{k-\kappa}\MP \varepsilon Q^{-\kappa}. \]
Since $Q \PS 1$ by the first part of~\eqref{condlin}, we can apply Proposition~\ref{dio}: there exist $d$ periodic vectors $v_1, \dots, v_d$, of periods $T_1, \dots, T_d$, such that $T_1v_1, \dots, T_dv_d$ form a $\Z$-basis of $\Z^n \cap F$ and for $j\in\{1,\dots,d\}$,
\[ |\omega-v_j|\MP(T_j Q)^{-1}, \quad 1 \MP T_j \MP \Psi_\omega(Q) .\] 

We claim that for all $1 \leq j \leq d$, there exists a symplectic map $\Phi^j \in C^{k-\kappa-1}(\mathcal{D}_{R_{d\kappa}-j\delta},\mathcal{D}_{R_{d\kappa}})$ such that 
\[ H_\kappa\circ\Phi^j=l_\omega+[f]_\omega+g_\kappa+[f_\kappa]_{v_1,\dots,v_j}+f_\kappa^j\]
with the estimates
\[ |\Phi^j-\mathrm{Id}|_{k-\kappa-1} \MP Q^{-1}, \quad |f_\kappa^j|_{k-\kappa-1} \MP \varepsilon Q^{-\kappa-1}.  \]
Assuming this claim, we let $\Phi_{\kappa+1}=\Phi_\kappa \circ \Phi^d$ so that
\[ H \circ \Phi_{\kappa+1} = l_\omega + [f]_\omega+g_{\kappa+1} + f_{\kappa+1} \]
with $g_{\kappa+1}=g_\kappa+[f_\kappa]_{v_1,\dots,v_d}$ and $f_{\kappa+1}=f_\kappa^d$. Then, as $R_{d\kappa}-d\delta=R_{d(\kappa+1)}$, we have $\Phi_{\kappa+1} \in C^{k-\kappa-1}(\mathcal{D}_{R_{d(\kappa+1)}},\mathcal{D}_R)$, and we can estimate
\[ |\Phi_{\kappa+1}-\mathrm{Id}|_{k-\kappa-1} \leq |\Phi_\kappa \circ \Phi^d-\Phi^d|_{k-\kappa-1}+|\Phi^d-\mathrm{Id}|_{k-\kappa-1} \]  
and as 
\[ |\Phi_\kappa \circ \Phi^d-\Phi^d|_{k-\kappa-1} \MP |\Phi_\kappa-\mathrm{Id}|_{k-\kappa-1} \]
by the estimate~\eqref{faa1} from Appendix~\ref{app1}, we obtain
\[ |\Phi_{\kappa+1}-\mathrm{Id}|_{k-\kappa-1} \MP |\Phi_\kappa-\mathrm{Id}|_{k-\kappa-1}+|\Phi^d-\mathrm{Id}|_{k-\kappa-1} \MP Q^{-1}. \]
Moreover, from Corollary~\ref{cordio}, $[f_\kappa]_{v_1,\dots,v_d}=[f_\kappa]_{\omega}$ and therefore $\{g_{\kappa+1},l_{\omega}\}=\{g_{\kappa},l_{\omega}\}+\{[f_\kappa]_{\omega},l_{\omega}\}=0$. Concerning the estimates for $g_{\kappa+1}$, we have to distinguish whether $1\leq \kappa \leq k-1$ or $\kappa=0$. In the first case, we have
\[ |g_{\kappa+1}|_{k-\kappa}\MP |g_{\kappa}|_{k-\kappa}+|[f_\kappa]_\omega|_{k-\kappa} \MP (\varepsilon Q^{-1}+\varepsilon Q^{-\kappa}) \MP \varepsilon Q^{-1}  \]
while in the second case, $g_1=0$, so that the above estimate is also true. To see that $g_1=0$, recall that $g_1=g_0+[f_0]_\omega$, but on the one hand, $g_0=0$, and the other hand, $f_0=f-[f]_\omega$ hence $[f_0]_\omega=0$ and so $g_1=0$. Therefore the statement remains true for $\kappa+1$, provided the claim holds true.

So to conclude the proof of the proposition, we need to prove the claim and we will proceed by induction on $1 \leq j \leq d$. First, for $j\in\{1,\dots,d\}$, let us define 
\[ \tilde{s}_j=l_\omega-l_{v_j}, \quad \nu_j \EP (T_j Q)^{-1} \]
with a suitable implicit constant so that $|\tilde{s}_j|_{\alpha,L} \leq \nu_j$. Note that $l_\omega=l_{v_j}+\tilde{s}_j$, and that $T_j\nu_j \EP Q^{-1}$. Observe that \eqref{condlin} implies \eqref{condana} for any $1 \leq j \leq d$: indeed, $Q \PS 1$ implies $T_j\nu_j \EP Q^{-1} \MP 1$, whereas 
\[ \varepsilon \MP \Delta_\omega(Q)^{-1} \EP (Q\Psi_\omega(Q))^{-1}\MP (T_jQ)^{-1} \EP \nu_j. \] 
For $j=1$, the statement follows from Proposition~\ref{ana}. Indeed, we can write
\[ H_\kappa=l_{v_1}+s_{1}+f_\kappa, \quad s_1=\tilde{s_1}+[f]_\omega+g_\kappa \]
and since $Q \PS 1$ and $\varepsilon \MP \nu_1$ we have the estimate
\[ |s_1|_{k-\kappa} \MP (\nu_1+\varepsilon+\varepsilon Q^{-1}) \MP \nu_1. \]
Recall that $|f_\kappa|_{k-\kappa} \MP \varepsilon Q^{-\kappa}$, and as $T_1\nu_1 \MP 1$ and $\varepsilon Q^{-\kappa} \leq \varepsilon \MP \nu_1$, we can therefore apply Proposition~\ref{ana} to the Hamiltonian $H_\kappa$, defined and of class $C^{k-\kappa}$ on $\mathcal{D}_{R_{d\kappa}}$, with $i=k-\kappa$, $v=v_1$, $s=s_1$ and $u=f_\kappa$, and we find a symplectic map $\Theta_1 \in C^{k-\kappa-1}(\mathcal{D}_{R_{d\kappa}-\delta},\mathcal{D}_{R_{d\kappa}})$ such that
\[ H_\kappa\circ\Theta_1=l_{v_1}+s_1+[f_\kappa]_{v_1}+f_\kappa' \]
with the estimates
\[ |\Theta_1-\mathrm{Id}|_{k-\kappa-1} \MP T_1\varepsilon \MP T_1\nu_1 \PE Q^{-1} \] 
and
\[ |f_\kappa'|_{k-\kappa-1} \MP \varepsilon Q^{-\kappa} T_1\nu_1 \PE \varepsilon Q^{-\kappa} Q^{-1} =  \varepsilon Q^{-\kappa-1}.  \]   
So we define
\[ \Phi^1=\Theta_1, \quad f^1_\kappa=f_\kappa' \]
to obtain
\[ H_\kappa\circ\Phi^1=l_{v_1}+\tilde{s}_1+[f]_\omega+g_\kappa+[f_\kappa]_{v_1}+f_\kappa^1=l_\omega+[f]_\omega+g_\kappa+[f_\kappa]_{v_1}+f_\kappa^1.\]
Hence the statement for $j=1$ is true. 

So now assume the statement holds true for some $1 \leq j \leq d-1$, and let us prove it is true for $j+1$. By the inductive assumption, there exists a symplectic map $\Phi^j \in C^{k-\kappa-1}(\mathcal{D}_{R_{d\kappa}-j\delta},\mathcal{D}_{R_{d\kappa}})$ such that 
\[ H_\kappa\circ\Phi^j=l_\omega+[f]_\omega+g_\kappa+[f_\kappa]_{v_1,\dots,v_j}+f_\kappa^j\]
with the estimates
\[ |\Phi^j-\mathrm{Id}|_{k-\kappa-1} \MP Q^{-1}, \quad |f_\kappa^j|_{k-\kappa-1} \MP \varepsilon Q^{-\kappa-1}.  \]
We can write
\begin{eqnarray*}
H_\kappa\circ\Phi^j - f_\kappa^j & = & l_\omega+[f]_\omega+g_\kappa+[f_\kappa]_{v_1,\dots,v_j} \\
& = & l_{v_{j+1}}+\tilde{s}_{j+1}+[f]_\omega+g_\kappa+[f_\kappa]_{v_1,\dots,v_j} \\
& = & l_{v_{j+1}}+s_{j+1}+[f_\kappa]_{v_1,\dots,v_j}
\end{eqnarray*}
where $s_j=\tilde{s}_j+[f]_\omega+g_\kappa$. The point is that even though both $H_\kappa\circ\Phi^j$ and $f_\kappa^j$ are of class $C^{k-\kappa-1}$, their difference is of class $C^{k-\kappa}$, as one easily sees from the expressions above. So exactly as for $j=1$, we can check that Proposition~\ref{ana} can be applied to the Hamiltonian $H_\kappa\circ\Phi^j - f_\kappa^j$, defined and of class $C^{k-\kappa}$ on $\mathcal{D}_{R_{d\kappa}-j\delta}$, still with $i=k-\kappa$ but this time with $v=v_{j+1}$, $s=s_{j+1}$ and $u=[f_\kappa]_{v_1,\dots,v_j}$, and we find a symplectic map $\Theta_{j+1} \in C^{k-\kappa-1}(\mathcal{D}_{R_{d\kappa}-(j+1)\delta},\mathcal{D}_{R_{d\kappa}-j\delta})$ such that
\[ (H_\kappa \circ \Phi^j-f_\kappa^j)\circ\Theta_{j+1}=l_{v_{j+1}}+s_{j+1}+[f_\kappa]_{v_1,\dots,v_{j+1}}+[f_\kappa]_{v_1,\dots,v_j}' \]
with the estimates
\[ |\Theta_{j+1}-\mathrm{Id}|_{k-\kappa-1} \MP T_{j+1}\varepsilon \MP T_{j+1}\nu_{j+1} \PE Q^{-1} \] 
and
\[ |[f_\kappa]_{v_1,\dots,v_j}'|_{k-\kappa-1} \MP \varepsilon Q^{-\kappa} T_{j+1}\nu_{j+1} \PE \varepsilon Q^{-\kappa} Q^{-1} =  \varepsilon Q^{-\kappa-1}.  \]    
So we define
\[ \Phi^{j+1}=\Phi^j \circ \Theta_{j+1}, \quad f_{\kappa}^{j+1}=[f_\kappa]_{v_1,\dots,v_j}'+f_\kappa^j \circ \Phi^{j+1} \]
to obtain
\[ H_\kappa\circ\Phi^{j+1}=l_\omega+[f]_\omega+g_\kappa+[f_\kappa]_{v_1,\dots,v_{j+1}}+f_\kappa^{j+1}.\]
We have $\Phi^{j+1} \in C^{k-\kappa-1}(\mathcal{D}_{R_{d\kappa}-(j+1)\delta},\mathcal{D}_{R_{d\kappa}})$, and exactly as for $\Phi_{\kappa+1}$, using the estimate~\eqref{faa1} from Appendix~\ref{app1}, we obtain
\[ |\Phi^{j+1}-\mathrm{Id}|_{k-\kappa-1} \MP  |\Phi^j-\mathrm{Id}|_{k-\kappa-1}+|\Theta_{j+1}-\mathrm{Id}|_{k-\kappa-1} \MP Q^{-1} \]
and also
\[ |f_{\kappa}^{j+1}|_{k-\kappa-1} \MP |[f_\kappa]_{v_1,\dots,v_j}'|_{k-\kappa-1} + |f_\kappa^j|_{k-\kappa-1} \MP \varepsilon Q^{-\kappa-1}. \]
So this completes the induction on $1\leq j \leq d$, which itself completes the induction on $0 \leq \kappa \leq k-1$, and therefore this ends the proof of the proposition.

\end{proof}

\bigskip

{\it Acknowledgments.} I am grateful to Vadim Kaloshin and Marcel Guardia for asking a question related to this work and which was the initial motivation, and to Jean-Pierre Marco for asking already a long time ago to prove upper bounds for the splitting of invariant manifolds for non-analytic systems. I am much indebted to Stéphane Fischler, as this work crucially uses results from our previous joint work. This paper has been written while the author was a Member at the Institute for Advanced Study in Princeton and a IPDE laureate at the Centre de Recerca Matemàtica in Barcelona, I thank both institutions for their support.

\appendix

\section{Technical estimates}\label{app1}

Let us begin by recalling some elementary estimates. First if $f\in C^k(\mathcal{D}_R)$, then for $|l|\leq j$, $\partial^l f \in C^{k-j}(\mathcal{D}_R)$ and obviously
\begin{equation}\label{eg1}
|\partial^l f|_{k-j}\leq |f|_k.
\end{equation}
In particular, this implies that if $f\in C^k(\mathcal{D}_R)$, then its Hamiltonian vector field $X_f$ is of class $C^{k-1}$  and  
\[ |X_f|_{k-1}\leq |f|_k. \]
Then, given two functions $f,g\in C^k(\mathcal{D}_R)$, the product $fg$ belongs to $C^k(\mathcal{D}_R)$ and by the Leibniz formula
\begin{equation}\label{eg2}
|fg|_{k} \MP |f|_{k}|g|_{k}.
\end{equation}
By~\eqref{eg1} and~\eqref{eg2}, the Poisson Bracket $\{f,g\}$ belongs to $C^{k-1}(\mathcal{D}_R)$ and
\begin{equation}\label{poisson}
|\{f,g\}|_{k-1} \MP |f|_{k}|g|_{k}.
\end{equation}
The above implicit constants depend only on $n$ and $k$.

We shall also use many times estimates which follows from the formula of Faà di Bruno (see \cite{AR67} for example), which gives bounds of the form
\begin{equation}\label{faa1}
|F\circ G|_{k} \MP |F|_{k}|G|^{k}_{k}
\end{equation}
where $F,G$ are vector-valued functions $F \in C^k(\mathcal{D}_{R_2},\mathcal{D}_{R_3})$ and $G \in C^k(\mathcal{D}_{R_1},\mathcal{D}_{R_2})$, for some positive numbers $R_1,R_2$ and $R_3$. This formula also gives the following estimate
\begin{equation}\label{faa2}
|F\circ G|_{k} \MP |F|_{1}|G|^{k}_{k}+|F|_{k}|G|^{k}_{k-1} 
\end{equation}
that we will use below.

Now if $f\in C^k(\mathcal{D}_R)$, $k\geq 2$, the Hamiltonian vector field $X_f$ is of class $C^{k-1}$ and so is the time-$t$ map $X_f^t$ of the vector field $X_f$, when it exists. For a given $0<\delta<1$, assuming that $|X_f|_{k-1}$ is small enough with respect to $\delta$ and $R$, it follows from general results on ordinary differential equations that for $|t|\leq 1$,  
\[ X_{f}^{t} : \mathcal{D}_{R-\delta} \longrightarrow \mathcal{D}_{R} \]
is a well-defined $C^{k-1}$-embedding.

We will need to prove that the $C^{k-1}$ norm of the distance of $X_{f}^{1}$ to the identity is bounded (up to constants depending on $k$ and on the domain) by the $C^{k-1}$ norm of the vector field $X_f$ (which itself is bounded by the $C^k$ norm of $f$). Let us state this as a lemma, the proof of which is a simple adaptation of Lemma $3.15$ in \cite{DH09}.

\begin{lemma}\label{tech} 
Let $f\in C^{k}(\mathcal{D}_R)$, $0<\delta<1$, and assume that
\begin{equation}\label{condtech}
|X_f|_{k-1} \MP 1.  
\end{equation}
Then, with the previous notations, we have 
\[ |X_f^1-\mathrm{Id}|_{k-1} \MP |X_f|_{k-1}. \]
\end{lemma} 

Note that the implicit constant depends on $k$, $R$ and $\delta$, and that we will use this statement for a value of $\delta$ depending only on $d$, $R$ and $k$.

\begin{proof}
We have the relation
\begin{equation}\label{rel}
X_f^t=\mathrm{Id}+\int_{0}^{t}X_f \circ X_f^s ds,
\end{equation}
hence by the estimate~\eqref{faa1}, is is enough to prove that for any $0\leq s \leq 1$, we have
\begin{equation}\label{est1}
|X_f^s|_{k-1} \MP 1. 
\end{equation}

Let $j=k-1$, and for $1 \leq i \leq j$, let us write 
\[ a_i=\sup_{0\leq s \leq 1}|X_f^s|_i, \quad b_i=|X_f|_{i}. \]
Using~\eqref{rel} and the estimate~\eqref{faa2}, we easily obtain
\[ a_1 \MP (1+b_1a_1)  \]
and
\[ a_i \MP (1+b_1a_i+b_ia_{i-1}^i). \]
Now using~\eqref{condtech}, $b_i \MP 1$ for $1 \leq i \leq j$, therefore by induction we have $a_i \MP 1$ for $1 \leq i \leq j$, and $a_j \MP 1$ implies~\eqref{est1}.  
\end{proof}

\addcontentsline{toc}{section}{References}
\bibliographystyle{amsalpha}
\bibliography{GNFC2}

\providecommand{\bysame}{\leavevmode\hbox to3em{\hrulefill}\thinspace}
\providecommand{\MR}{\relax\ifhmode\unskip\space\fi MR }
\providecommand{\MRhref}[2]{%
  \href{http://www.ams.org/mathscinet-getitem?mr=#1}{#2}
}
\providecommand{\href}[2]{#2}
\begin{thebibliography}{Bou12b}

\bibitem[AR67]{AR67}
R.~Abraham and J.~Robbin, \emph{Transversal mappings and flows}, Benjamin,
  New-York, 1967.

\bibitem[BF12]{BF12}
A.~Bounemoura and S.~Fischler, \emph{A diophantine duality applied to the {KAM}
  and {N}ekhoroshev theorems}, submitted.

\bibitem[BN12]{BN09}
A.~Bounemoura and L.~Niederman, \emph{Generic {N}ekhoroshev theory without
  small divisors}, Ann. Inst. Fourier \textbf{62} (2012), no.~1, 277--324.

\bibitem[Bou10]{Bou10}
A.~Bounemoura, \emph{Nekhoroshev theory for finitely differentiable
  quasi-convex {H}amiltonians}, Journal of Differential Equations \textbf{249}
  (2010), no.~11, 2905--2920.

\bibitem[Bou12a]{BouI12}
\bysame, \emph{Normal forms, stability and splitting of invariant manifolds
  {I}. {G}evrey {H}amiltonians}, preprint.

\bibitem[Bou12b]{Bou12}
\bysame, \emph{Optimal stability and instability for near-linear hamiltonians},
  Annales Henri Poincaré \textbf{13} (2012), no.~4, 857--868.

\bibitem[DH09]{DH09}
A.~Delshams and G.~Huguet, \emph{Geography of resonances and {A}rnold diffusion
  in a priori unstable {H}amiltonian systems}, Nonlinearity \textbf{22} (2009),
  no.~8, 1997--2077.

\bibitem[LMS03]{LMS03}
P.~Lochak, J.-P. Marco, and D.~Sauzin, \emph{{On the splitting of invariant
  manifolds in multidimensional near-integrable {H}amiltonian systems}}, Mem.
  Am. Math. Soc. \textbf{775} (2003), 145 pp.

\end{thebibliography}

\end{document}